\documentclass{siamart220329}

\usepackage{arxiv}

\usepackage[utf8]{inputenc} 
\usepackage[T1]{fontenc}    
\usepackage{hyperref}       
\usepackage{url}            
\usepackage{booktabs}       
\usepackage{amsfonts}       
\usepackage{nicefrac}       
\usepackage{microtype}      
\usepackage{lipsum}
\usepackage{graphicx}
\graphicspath{ {./images/} }

\usepackage{rotating} 

\usepackage{scalerel}
\usepackage{tikz}
\usetikzlibrary{svg.path}

\definecolor{orcidlogocol}{HTML}{A6CE39}
\tikzset{
  orcidlogo/.pic={
    \fill[orcidlogocol] svg{M256,128c0,70.7-57.3,128-128,128C57.3,256,0,198.7,0,128C0,57.3,57.3,0,128,0C198.7,0,256,57.3,256,128z};
    \fill[white] svg{M86.3,186.2H70.9V79.1h15.4v48.4V186.2z}
                 svg{M108.9,79.1h41.6c39.6,0,57,28.3,57,53.6c0,27.5-21.5,53.6-56.8,53.6h-41.8V79.1z M124.3,172.4h24.5c34.9,0,42.9-26.5,42.9-39.7c0-21.5-13.7-39.7-43.7-39.7h-23.7V172.4z}
                 svg{M88.7,56.8c0,5.5-4.5,10.1-10.1,10.1c-5.6,0-10.1-4.6-10.1-10.1c0-5.6,4.5-10.1,10.1-10.1C84.2,46.7,88.7,51.3,88.7,56.8z};
  }
}

\newcommand\orcidicon[1]{\href{https://orcid.org/#1}{\mbox{\scalerel*{
\begin{tikzpicture}[yscale=-1,transform shape]
\pic{orcidlogo};
\end{tikzpicture}
}{|}}}}

\usepackage{bm} 
\newcommand{\bx}{{\bf x}}

\newcommand{\bX}{{\bf X}}

\newcommand{\ba}{{\bf a}}
\newcommand{\bA}{{\bf A}}
\newcommand{\bb}{{\bf b}}

\newcommand{\bC}{{\bf C}}
\newcommand{\bI}{{\bf I}}
\newcommand{\bK}{{\bf K}}

\newcommand{\bQ}{{\bf Q}}

\newcommand{\bmxi}{{\bm{\xi}}}

\newcommand{\Alim}{{A_{\text{LIM}}}}
\newcommand{\Acslim}{{A_{\textit{e}}}}
\newcommand{\Aplim}{{A_{\textit{l}}}}
\newcommand{\Qlim}{{Q_{\text{LIM}}}}
\newcommand{\Qcslim}{{Q_{\textit{e}}}}
\newcommand{\Qplim}{{Q_{\textit{l}}}}

\newcommand{\Cplim}{{C_{\textit{l}}}}
\newcommand{\Tcslim}{{T_{\textit{e}}}}
\newcommand{\Tplim}{{T_{\textit{l}}}}

\newcommand{\LFP}{\textbf{L}_{\textit{FP}}}
\newcommand{\pst}{P_{\textit{st}}}
\usepackage{algorithm,algpseudocode}
\algblock{Input}{EndInput}
\algnotext{EndInput}
\algblock{Output}{EndOutput}
\algnotext{EndOutput}
\newcommand{\Desc}[2]{\State \makebox[3em][l]{#1}#2}
\newcommand{\algrule}[1][.2pt]{\par\vskip.5\baselineskip\hrule height #1\par\vskip.5\baselineskip}

\usepackage{lipsum}
\usepackage{amsfonts}
\usepackage{graphicx}
\usepackage{epstopdf}
\ifpdf
  \DeclareGraphicsExtensions{.eps,.pdf,.png,.jpg}
\else
  \DeclareGraphicsExtensions{.eps}
\fi

\newcommand{\R}{\mathbb{R}}

\newsiamremark{remark}{Remark}
\newsiamremark{hypothesis}{Hypothesis}
\crefname{hypothesis}{Hypothesis}{Hypotheses}
\newsiamthm{claim}{Claim}

\usepackage{amsmath} 
\usepackage{mathtools} 
\usepackage{booktabs,graphicx,makecell,siunitx,eqparbox,subcaption,threeparttable}
\usepackage{adjustbox}
\usepackage{lscape} 
\usepackage{afterpage}
\usepackage{graphicx} 
\usepackage{array}

\usepackage{tikz}
\usetikzlibrary{arrows.meta}
\tikzset{mynode/.style={draw, very thick, circle, minimum size=0.8cm}, myarrow/.style={very thick}}

\usepackage{graphicx}

\usepackage{hyperref} 

\title{On the cyclostationary Linear Inverse Models: a mathematical insight and implication}

\author{
    Justin Lien \\
    Mathematical Institute \\
    Tohoku University\\
    Sendai, Japan \\
    \texttt{lien.justin.t8@dc.tohoku.ac.jp} \\
    \And
    Yan-Ning Kuo \\
    Department of Earth and Atmospheric Sciences\\
    Cornell University\\
    Ithaca, New York, USA \\
    \texttt{yk545@cornell.edu} \\
    \And
    Hiroyasu Ando \\
    Advanced Institute for Materials Research\\
    Tohoku University\\
    Sendai, Japan \\
    \texttt{hiroyasu.ando.d1@tohoku.ac.jp} \\
}

\begin{document}
\maketitle
\begin{abstract}
Cyclostationary linear inverse models (CS-LIMs), generalized versions of the classical (stationary) LIM, are advanced data-driven techniques for extracting the first-order time-dependent dynamics and random forcing relevant information from complex non-linear stochastic processes. 
Though CS-LIMs lead to a breakthrough in climate sciences, their mathematical background and properties are worth further exploration.
This study focuses on the mathematical perspective of CS-LIMs and introduces two variants: $e$-CS-LIM and $l$-CS-LIM. The former refines the original CS-LIM using the interval-wise linear Markov approximation, while the latter serves as an analytic inverse model for the linear periodic stochastic systems. 
Although relying on approximation, $e$-CS-LIM converges to $l$-CS-LIM under specific conditions and shows noise-robust performance.

Numerical experiments demonstrate that each CS-LIM reveals the temporal structure of the system. 
The $e$-CS-LIM optimizes the original model for better dynamics performance, while $l$-CS-LIM excels in diffusion estimation due to reduced approximation reliance. 
Moreover, CS-LIMs are applied to real-world ENSO data, yielding a consistent result aligning with observations and current ENSO understanding.
\end{abstract}


\section{Introduction} \label{Chap:Intro}

Stochastic differential equation (SDE) is a mathematical framework employed to study dynamical systems subjected to both deterministic and stochastic influences. 
It combines the deterministic part expressed through ordinary differential equations with the random forcing term formulated by Gaussian white noise, making itself invaluable across diverse disciplines \cite{Oksendal1987,sarkka_solin_2019}. 
For example, SDEs are utilized to capture the unpredictable nature of stock prices in financial mathematics and describe the erratic movement of particles in a fluid in physics \cite{Oksendal1987,Ruschendorf2023,Seifert2012}.
Furthermore, apart from modeling, SDEs also find applications in inverse problems, including the classical linear inverse model (LIM) \cite{Kwasniok2022,Penland1989,Penland1993,Penland1994}.

The classical LIM serves as a mathematical tool that extracts the linear dynamics and random forcing behavior of the underlying complex non-linear stochastic process from finite sampling data, allowing scientists to infer the underlying network dynamics and quantify uncertainties that are challenging to measure directly \cite{Penland1993,Penland1994,Penland1995}.
More precisely, consider a dynamical system of the form 
\begin{align} \label{Eq:GeneralEq}
    \frac{d}{dt} \bx = f(\bx,t,\bmxi),
\end{align}
where $f$ is the unknown system and $\bmxi$ represents the normalized Gaussian vector.  
The classical LIM approximates \cref{Eq:GeneralEq} with a linear Markov system
\begin{equation}\label{Eq:LIM-Process}
    \frac{d}{dt} \bx = \bA\bx + \sqrt{2\bQ} \bmxi,
\end{equation}
where $\bQ$ describes the covariance of noise.
We call this the linear Markov approximation.
As an Ornstein–Uhlenbeck process, the correlation function $\bK$ is an exponential function whose exponent is the constant dynamics $\bA$, which can be solved provided the values of $\bK$ at the origin and at some time lag are known.
Then, the constant diffusion matrix $\bQ$ can be obtained through the fluctuation-dissipation relation (FDR).
Therefore, both linear dynamics and diffusion can be estimated given that we have a finite realization of a stochastic process at hand.
Due to its mathematical simplicity and applicability, the classical LIM has been widely applied to climate sciences to study large-scale climate events including El Ni\~{n}o-Southern Oscillation (ENSO) \cite{Penland1993,Penland1994,Perkins2020}.
However, the classical LIM fails to reveal the seasonal variation which serves as a crucial element to understanding the complex climate system, and thus several follow-up variants of the classical LIM have been proposed \cite{Martinez2017,Penland1996}.

The original cyclostationary linear inverse model (CS-LIM), as a variant of LIMs, is a model extracting the temporal structure of a periodically driven stochastic process.
That is, the unknown system $f$ in \cref{Eq:GeneralEq} is periodic.
It first appeared in the climate science community to study the seasonal variation of ENSO through monthly sea surface temperature (SST) time-series data, showing an improved forecast skill and a more accurate ENSO characteristic compared to the analysis based on the classical LIM \cite{Kido2023,OrtizBevia1997,Shin2021}.
Though the original version has proven an effective data-driven technique in practical applications, the mathematical formulation and properties are worth further exploration.

From a mathematical perspective, under the stationary condition, the original CS-LIM first approximates the complex non-linear stochastic system by the periodic linear Markov system of the form
\begin{equation}\label{Eq:CS-LIM-Process}
    \frac{d}{dt} \bx = \bA(t)\bx + \sqrt{2\bQ(t)} \bmxi,
\end{equation}
where $\bA(t)$ and $\bQ(t)$ are periodic families of dynamical and diffusion matrices.
To estimate $\bA(t)$ and $\bQ(t)$, the original CS-LIM divides the full period into several intervals, applies the classical LIM interval-wise to extract the linear dynamics (interval-wise linear Markov approximation), and then estimates the diffusion by enforcing the periodic version of FDR. 
Therefore, unlike the classical LIM that reconstructs the linear dynamics and diffusion of a linear Markov system, the original CS-LIM does not reconstruct the periodic linear dynamics and diffusion for a process satisfying \cref{Eq:CS-LIM-Process} but rather estimates them.
Even so, it is sufficient for most of the applications.
The detailed formulation of \cref{Eq:LIM-Process,Eq:CS-LIM-Process} will be given in the later section.


In this article, we study LIMs from a mathematical perspective by providing a solid theoretical background and examining the validity of a various approximation used in the models.
In particular, we propose the $e$-CS-LIM, a variant of CS-LIM that follows the same framework as the original CS-LIM: estimating the dynamics by the interval-wise linear Markov approximation and the diffusion by periodic FDR.
However, the numerical detail is refined such that $e$-CS-LIM optimizes the original CS-LIM.
Moreover, we present a novel CS-LIM called $l$-CS-LIM that serves as an inverse model to \cref{Eq:CS-LIM-Process} by proving that the time-dependent dynamics is encoded in the first right derivative of the correlation function.
We also note that it amounts to the pointwise linear Markov approximation, contrary to the interval-wise version.

In fact, the $e$-CS-LIM and $l$-CS-LIM are closely related.
In practical implementation, the former utilizes an \textit{exponential fitting} while the latter uses a \textit{linear fitting} when estimating the linear dynamics of a general periodic stochastic process. 
The exponential fitting approaches to the linear fitting when the time-step and the time-lag are sufficiently small, and the number of intervals and the sampling size are sufficiently large.
Therefore, in such a limit, $e$-CS-LIM converges to $l$-CS-LIM in the sense that the estimated results can be arbitrarily small.
In theory, though they are both first-order models of a complex non-linear stochastic process, in practice, they are skilled at different aspects.
In the numerical experiment, we observed that owing to a larger time-step and time-lag applied in the computation, the $e$-CS-LIM is more robust to noise, leading to a more stable result. 
On the other hand, though being an analytic inverse model, $l$-CS-LIM is subject to the noisy nature of the SDE and a denoise process should be done.
Nevertheless, the $e$-CS-LIM exhibits a comparable or better accuracy than $l$-CS-LIM in linear dynamics, while the $l$-CS-LIM performs better in diffusion, providing that the sampling data is sufficient. 

The structure of the article goes as follows. In \cref{Chap:LIM}, we review the mathematical backgrounds of the classical LIM and introduce the $e$-CS-LIM.
The original CS-LIM will be a special case of $e$-CS-LIM in our framework.
In \cref{Chap:Periodic-LIM}, we develop the $l$-CS-LIM by studying the correlation function of \cref{Eq:CS-LIM-Process}, and discuss its relationship with the classical LIM and $e$-CS-LIM.
Then, the numerical experiments are presented in \cref{Chap:NumExp} to demonstrate the potential issues of the classical LIM and the original CS-LIM, the effectiveness of $e$-CS-LIM and $l$-CS-LIM, and their application. 

Before moving to the next section, we briefly explain the convention of our notation. 
The vector is assumed to be a column vector without explicitly stated. 
For continuous-time stochastic processes, random variables and their statistics are denoted in bold, and so is the dynamical matrix for the notation consistency. When referring to time-series data, we mean a discrete-time sequence of vectors with an equal sampling interval represented by $\Delta t$, and we use the regular font to denote both the time series and its statistics. 
Moreover, the true values are denoted in bold while the model outputs are in regular. 

\section{\textit{e}-CS-LIM} \label{Chap:LIM}

In this section, we present the mathematical background and idea of $e$-CS-LIM with a brief review of the classical LIM.
As all variants of LIMs utilize FDRs, we start by introducing the Fokker-Planck equation, which in turn characterizes the probability distribution of the system under stationary conditions. 
The proof can be found in standard SDE textbooks \cite{Oksendal1987,sarkka_solin_2019}.

Suppose that the stochastic process $\bx: [0,\infty) \to \R^n$ satisfies the linear dynamics with Gaussian noises random forcing as follows,
\begin{equation}\label{Eq:General-Process}
    \frac{d}{dt} \bx = \bA(t)\bx + \sqrt{2\bQ(t)} \bmxi,
\end{equation}
where $\bA(t)$ and $\bQ(t) \in \R^{n \times n}$ are $C^1$-families of dynamical matrices and (positively definite) diffusion matrices, respectively; $\bmxi = (\bmxi_1,\dots,\bmxi_n)^T \in \R^n$ is the normalized Gaussian random vector with zero mean and satisfies
\begin{equation}\label{Eq:NormalizedGaussian}
    \langle \bmxi(t)\bmxi^T(s) \rangle = \delta(t-s)\bI,
\end{equation} 
where the bracket, $\delta(t-s)$, and $\bI$ denote the expectation, the Dirac delta function, and the identity matrix, respectively.

The time evolution of the probability distribution of \cref{Eq:General-Process} is given by the following theorem.
\begin{theorem}[Fokker-Plank equation]
    With the notation as above, the probability distribution $P(x,t) = \langle \delta(\bx(t)-x)\rangle$ of the stochastic process $\bx$ satisfies
    \begin{align} \label{Eq:FP-eq}
        \frac{\partial}{\partial t} P(x,t) &= - \sum_{j,k} \bA_{jk}(t) \frac{\partial}{\partial x_j} x_k P(x,t) + \sum_{j,k} \bQ_{jk}(t) \frac{\partial^2}{\partial x_j \partial x_k} P(x,t) \nonumber \\ 
        &= \LFP(t) P(x,t),
    \end{align}
    where $\LFP(t)$ is the Fokker-Planck operator at $t$.  
\end{theorem}

In this study, we further require that the dynamical function $\bA(t)$ is such that the Fokker-Planck equation \cref{Eq:FP-eq} admits a (periodic) stationary probability distribution.
Though perhaps too restrictive in practice, an example is that each eigenvalue of the dynamical matrix $\bA(t)$ has a negative real part for any time $t$. 

\subsection{The classical LIM} 
The classical LIM approximates a complex dynamical system \cref{Eq:GeneralEq} by a linear Markov model \cref{Eq:LIM-Process} (i.e., constant linear dynamics $\bA(t) \equiv \bA$ and diffusion $\bQ(t) \equiv \bQ$) \cite{Penland1989,Penland1994}.
In the steady state, the statistics of the process $\bx$ are independent of time, and the Fokker-Planck equation \cref{Eq:FP-eq} leads to the following \cite{Lien2024,Penland1994}.

\begin{corollary}[The classical fluctuation-dissipation relation] \label{Cor:FD-relation}
    Let the covariance matrix be $\bC \coloneqq \langle \bx(\cdot) \bx(\cdot)^T \rangle$.
    For a linear Markov system, in the steady state, we have
    \begin{equation} \label{Eq:FD-relation}
        0 = \bA\bC + \bC \bA^T + 2 \bQ.
    \end{equation}
\end{corollary}

The classical FDR \cref{Eq:FD-relation} connects the linear dynamics and diffusion matrix provided that the covariance matrix is known a prior.
Indeed, the following theorem establishes a one-to-one correspondence between the dynamical matrix and the correlation function of a linear Markov system.
Hence, both $\bA$ and $\bQ$ can be inferred once the correlation function is known \cite{Lien2024,Penland1994}.

\begin{theorem}
    Let the correlation function be given by $\bK(s) \coloneqq \langle \bx(\cdot+s) \bx(\cdot)^T \rangle$. For a linear Markov system \cref{Eq:LIM-Process}, the dynamical matrix $\bA$ satisfies
    \begin{equation} \label{Eq:GreenFunc}
        \bA = \frac{1}{s} \log\big( \bK(s) \bK(0)^{-1} \big)
    \end{equation}
    for any $s>0$, where the $\log$ denotes the matrix logarithm. In fact, the correlation function is the exponential function of the form
    \begin{equation} \label{Eq:LIM-corr-func}
        \bK(s) = e^{\bA \left\vert s \right\vert} \bK(0). 
    \end{equation} 
\end{theorem}

In practice, given a time-series data $\{x(t): t = 0, \Delta t, \dots, N\Delta t\} \subset \R^n$, we assume the stationary condition. Then, the correlation function is numerically computed by
\begin{equation} \label{Eq:DiscCorrFuc}
    K(s) = \frac{\sum_{t=0}^{(N-k)\Delta t} x(t+s)x(t)^T}{N-k+1},
\end{equation} 
as in the steady state, the ensemble average is equal to the time average.
Then the constant dynamics $\Alim$ and diffusion $\Qlim$ are estimated based on \cref{Eq:FD-relation,Eq:GreenFunc}


\subsection{The periodic fluctuation-dissipation relation} \label{Chap:P-FD-relation}
Next, we move to a periodic linear Markov system \cref{Eq:CS-LIM-Process}. Without loss of generality, we may assume that the linear dynamics and diffusion are $1$-periodic; that is, $\bA(t+1) = \bA(t)$ and $\bQ(t+1) = \bQ(t)$.
Let $\pst$ denote the periodic solution of \cref{Eq:FP-eq}; that is, $\partial_t \pst = \LFP(t) \pst$ where $\pst(x,t) = \pst(x,t+1)$. 
Hence, the covariance function $\bC(t) \coloneqq \langle \bx(t) \bx(t)^T \rangle$ is also $1$-periodic, and the classical FDR can be extended to a periodic version. 

\begin{theorem}
    The periodic fluctuation-dissipation relation reads 
    \begin{equation} \label{Eq:Periodic-FD-Relation}
        \frac{d\bC}{dt} = \bA(t) \bC + \bC \bA^T(t) + 2\bQ(t).
    \end{equation}
\end{theorem}

\begin{proof}
    The adjoint Fokker-Planck operator $\LFP^*(t)$ for each $t$ is 
    \begin{equation*}
        \LFP^*(t) = \sum_{i,j} \bA_{ij}(t) x_j\frac{\partial}{\partial x_i} + \sum_{i,j} \bQ_{ij}(t) \frac{\partial^2}{\partial x_i \partial x_j}.
    \end{equation*}
    A direct computation shows
    \begin{align*}
        \LFP^*(t) x_px_q &= \sum_j \bA_{pj}(t) x_j x_q + \sum_j \bA_{qj}(t) x_j x_p + \bQ_{pq}(t). 
    \end{align*}
    Therefore, we have
    \begin{align*}
        \frac{\partial}{\partial t} \int_{\R^n} x_p x_q \pst(x,t) \,dx &= \int_{\R^{n}} x_p x_q \LFP(t) \pst(x,t) \,dx \\
        &= \int_{\R^{n}} \pst(x,t) \LFP^*(t) x_p x_q \,dx \\
        &= \int_{\R^{n}} \pst(x,t) \big( \sum_j \bA_{pj}(t) x_j x_q + \sum_j \bA_{qj}(t) x_j x_p+ \bQ_{pq}(t) \big) \,dx \\
        &= \sum_j \bA_{pj}(t) \langle \bx_j \bx_q \rangle + \sum_j \bA_{qj}(t) \langle \bx_j \bx_p \rangle + \bQ_{pq}(t),
    \end{align*}
    which is equivalent to \cref{Eq:Periodic-FD-Relation} in matrix notation.
\end{proof}

\subsection{Algorithm} \label{Chap:CS-LIM}
The $e$-CS-LIM estimates the time-dependent linear dynamics and diffusion of a general periodic stochastic process by approximating the underlying system with a periodic linear Markov system \cite{Shin2021}.
We present an overview of the fundamental concept.
First, a full period is partitioned into several intervals, and then within each interval, the linear Markov approximation is made so that the classical LIM can be applied to estimate the linear dynamics; finally, the diffusion is estimated via the periodic FDR \cref{Eq:Periodic-FD-Relation}.
Though the interval-wise linear Markov approximation works decently well in practice, we point out that $e$-CS-LIM is a hybrid-type model in the sense that the first half is based on approximation while the second half relies on the analytic formula; hence, it does not analytically reconstruct the periodic linear dynamics and diffusion of the linear stochastic process \cref{Eq:CS-LIM-Process}, as mentioned in \cref{Chap:Intro}.


The pseudo-code of $e$-CS-LIM is provided in \cref{Alg:CS-LIM}.
For notational convenience, whenever there is no confusion, we write $\partial_s$ for $\frac{\partial}{\partial s}$, we do not distinguish the vectorial index $j$ and the time $t = j\Delta t$, and the index should be taken modulo the full period.
As a special case, the original CS-LIM sets the time coordinate $T(j)$ for the estimated dynamical and diffusion matrices to be the center of the $j$-th interval, and the derivative of the covariance function is computed by the central difference method.
From now on, when referring $e$-CS-LIM, we use the forward difference method to compute the derivative of the covariance function and set the time coordinate $\Tcslim(j)$ to be the center of the $j$-th interval plus $\frac{1}{2} k \Delta t$ since the $j$-th dynamical matrix is computed by the $j$-th interval and its $k\Delta t$ lag (see \cref{Fig:Stencil}).
In \cref{Chap:NumExp}, we will see that the original CS-LIM introduces a phase shift while the $e$-CS-LIM picks up the correct phase due to the choice of the finite different scheme and the refined time coordinate (see also \cref{Fig:Periodic-Diffusion}). 

\begin{figure}[htbp]
\centering
    \begin{subfigure}[b]{0.49\textwidth}
    \centering
    \includegraphics[width=2.4in]{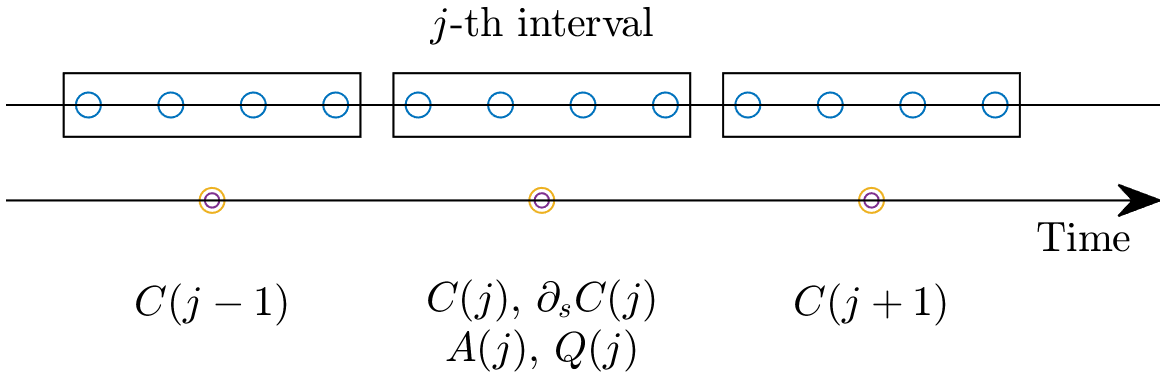}
    \caption{The original CS-LIM. \label{Fig:CS-LIM-Stencil}}
    \end{subfigure}
    \begin{subfigure}[b]{0.49\textwidth}
    \centering
    \includegraphics[width=2.4in]{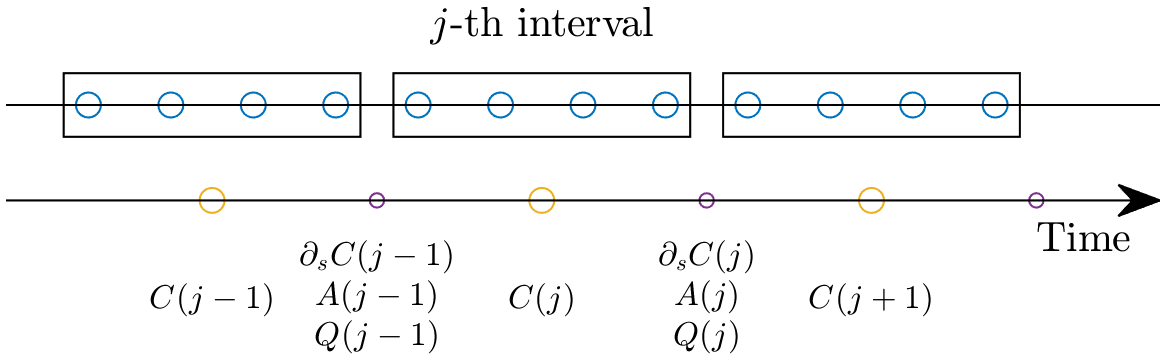}
    \caption{$e$-CS-LIM. \label{Fig:e-CS-LIM-Stencil}}
    \end{subfigure}
    \caption{ The stencil configurations. The blue dots on the top indicate the sampling time and the yellow and purple in the bottom are the time coordinates of the covariance function and the other three variables, respectively. \label{Fig:Stencil}}
\end{figure}

\begin{algorithm}
\caption{$e$-CS-LIM}\label{Alg:CS-LIM}

\begin{algorithmic}[1]

\Input
\Desc{$\{x(t)\}$}{A cyclostationary time-series data}
\Desc{$\Delta t$}{The sampling time step}
\Desc{$L$}{The number of periods}
\Desc{$M$}{The number of intervals}
\Desc{$k$}{The time lag}
\EndInput
\Output
\Desc{$\Acslim$}{The estimated linear dynamics}
\Desc{$\Qcslim$}{The estimated diffusion}
\Desc{$\Tcslim$}{The time coordinate for $\Acslim$ and $\Qcslim$}
\EndOutput
\algrule
\State Subdivide the full periodic into $M$ subintervals $\{I_j\}$. 
\For{$j = 0:M-1$}
\State Compute the correlation $K_j(0)$ and $K_j(k\Delta t)$ using $\{x(t)\}$ and $\{x(t+k\Delta t)\}$ where $t \in I_j$.
\State Compute $\Acslim(j)$ via \cref{Eq:GreenFunc}.
\State Compute the time $\Tcslim(j)$ that stands for the $j$-th interval.
\EndFor
\For{$j = 0:M-1$}
\State Compute $\Qcslim(j)$ via \cref{Eq:Periodic-FD-Relation}.
\EndFor

\end{algorithmic}
\end{algorithm}

\section{\textit{l}-CS-LIM} \label{Chap:Periodic-LIM}

As the $e$-CS-LIM \textit{interval-wise approximates} the underlying stochastic process by a linear Markov system, we aim to build an (analytic) inverse model for the periodic linear Markov system \cref{Eq:CS-LIM-Process}.
With the help of the periodic FDR, we are left to derive an explicit formula for the periodic dynamics from the statistics of $\bx$. 

\subsection{Mathematical background} \label{Chap:Periodic-LIM-Math}

In this subsection, we apply the setup and notation in \cref{Chap:P-FD-relation} and study the correlation function $\bK(s,t) \coloneqq \langle \bx(t+s) \bx^T(t) \rangle$, which is $1$-periodic in the time variable $t$.
The following theorem states that the linear dynamics is encoded in the local behavior of the correlation function at the origin of the lag variable $s$.

\begin{theorem}
    The first right derivatives of the correlation function with respect to the lag variable $s$ at the origin satisfies
    \begin{equation} \label{Eq:P-LIM-Dynamics}
        \frac{\partial}{\partial s} \bK(s,t)|_{s = 0} = \bA(t) \bC(t).
    \end{equation}
\end{theorem}

\begin{proof}
    
A direct computation shows that 
\begin{align*}
    \LFP^*(t) x_p &= \sum_j \bA_{pj}(t) x_j.
\end{align*}

Next, we compute the derivatives of the correlation function \cite{Jung1985,Risken1989}.
\begin{align*}
    \bK_{pq}(s,t) &= \int_{\R^{n}} x_p e^{ \int_t^{t+s} \LFP(u) \,du } \pst(x,t) x_q \,dx \\
    &= \int_{\R^{n}} x_q \pst(x,t) e^{ \int_t^{t+s} \LFP^*(u) \,du } x_p \,dx \\
    &= \int_{\R^{n}} x_q \pst(x,t) (\bI + s\LFP^*(t)) x_p + O(s^2) \,dx \\
    &= \langle \bx_p \bx_q \rangle(t) + s \int_{\R^{n}} x_q \pst(x,t) \LFP^*(t) x_p \,dx + O(s^2).
\end{align*}
Hence, the first derivative with respect to the lag variable $s$ is 
\begin{align*}
    \frac{\partial}{\partial s} \bK_{pq}(s,t) |_{s = 0} 
    &= \int_{\R^n} x_q \pst(x,t) \LFP^*(t) x_p \,dx \\
    &= \int_{\R^n} x_q \pst(x,t) \sum_j \bA_{pj}(t) x_j \,dx \\
    &= \sum_j \bA_{pj}(t) \langle \bx_j \bx_q \rangle,
\end{align*}
which is equivalent to \cref{Eq:P-LIM-Dynamics} in matrix notation.
\end{proof}


Like the previous LIMs, $l$-CS-LIM solves both $\bA(t)$ and $\bQ(t)$ from a given observation data based on \cref{Eq:Periodic-FD-Relation,Eq:P-LIM-Dynamics}. 
Moreover, $l$-CS-LIM is an analytic inverse model of \cref{Eq:CS-LIM-Process}.

\subsection{Algorithm}

In practice, given a cyclostationary time-series data $\{x(t)\}$, the $l$-CS-LIM approximates the underlying system by a periodic linear Markov system and estimates the linear dynamics and the random forcing, as $e$-CS-LIM does. 
The pseudo-code of the $l$-CS-LIM is given in \cref{Alg:P-LIM} in which for $\{x(t)\}$ consisting of $L$ full periods, the correlation function is calculated by
\begin{equation} \label{Eq:CorrFunc-s-t}
    K(s,t) = \frac{\sum_{k=0}^{L-1} x(k+t+s)x(k+t)^T}{L},
\end{equation} 
and its first right derivative in the lag variable $s$ is computed by the forward difference method.
\Cref{Eq:CorrFunc-s-t} approximates $\bK$ under the steady-state condition provided the sampling size is sufficiently large. 
Though $\bK$ is (one-sided) differentiable from a theoretical viewpoint, the numerical correlation function $K$ is noisy in both variables. 
In particular, preprocessing or postprocessing should be done to denoise.
We will discuss the numerical details in \cref{Chap:NumExp}. 

In this article, we compute $\partial_s C$ by the forward difference method and set the time coordinate of $\Aplim(t)$ and $\Qplim(t)$ to be in the middle of $t$ and $t+\Delta t$.
In fact, we remark that in our numerical experiment (\cref{Chap:1d-study,Chap:Higher-Dimension}), the choice of the time coordinate for $l$-CS-LIM is relatively minor significant.

\begin{algorithm}
\caption{$l$-CS-LIM}\label{Alg:P-LIM}
\begin{algorithmic}[1]
\Input
\Desc{$\{x(t)\}$}{A cyclostationary time-series data}
\Desc{$\Delta t$}{The sampling time step}
\Desc{$L$}{The number of periods}
\EndInput
\Output
\Desc{$\Aplim$}{The estimated linear dynamics}
\Desc{$\Qplim$}{The estimated diffusion}
\Desc{$\Tplim$}{The time label for $\Aplim$ and $\Qplim$}
\EndOutput
\algrule
\For{$t$ in $[0,1)$}
\State Compute $\Cplim(t) = K(0,t)$ and $K(\Delta t,t)$ by \cref{Eq:CorrFunc-s-t}.
\State Compute $\Aplim(t)$ via \cref{Eq:P-LIM-Dynamics}.
\State Compute $T_l(t)$.
\EndFor

\For{$t$ in $[0,1)$}
\State Compute $\Qplim(t)$ via \cref{Eq:Periodic-FD-Relation}.
\EndFor

\end{algorithmic}
\end{algorithm}

\subsection{The relationship among LIMs} \label{Chap:LIM-Relationships}

First, we note that the $l$-CS-LIM is indeed an extension of the classical LIM. 
If the linear dynamics and diffusion matrix are constant in time, the correlation function is independent of time variable $t$ and \cref{Eq:P-LIM-Dynamics} can be written as 
\begin{equation} \label{Eq:P-LIM-Dynamics-LIM}
    \frac{d}{ds} \bK(s) |_{s=0} = \bA \bC
\end{equation}
which is equivalent to the right derivative of \cref{Eq:LIM-corr-func}.
At the same time, for a general periodic stochastic process, we can view $l$-CS-LIM as approximating the underlying system with a linear Markov system at each instant as \cref{Eq:P-LIM-Dynamics,Eq:P-LIM-Dynamics-LIM} share the same form.
Therefore, we may say that $l$-CS-LIM utilizes the pointwise linear Markov approximation contrary to the interval-wise version adopted in $e$-CS-LIM.

From a practical viewpoint, at a given time $t$, the $l$-CS-LIM applies the forward difference method to the correlation function at two consecutive points $s = 0, \Delta t$, which amounts to fit these two points by a straight line. On the other hand, the $e$-CS-LIM calculates the correlation function at the origin and the point at $k$ time-step away, then applying \cref{Eq:GreenFunc}, which corresponds to an exponential curve fitting.
Since the linear fitting is the first-order approximation of the exponential fitting, the $e$-CS-LIM converges to the $l$-CS-LIM in the limit of $k=1$, $M=\frac{1}{\Delta t}$, $\Delta t \ll 1$ and $L \gg 1$.
As a result, the $l$-CS-LIM can be viewed as an instantaneous version of $e$-CS-LIM.
This also justifies the prefixes $e$-CS-LIM and $l$-CS-LIM.

Finally, it is natural to question whether the classical LIM accurately estimates the mean dynamics and diffusion of the periodic linear Markov system \cref{Eq:CS-LIM-Process}.
We observe that the correlation function used in the classical LIM is the integral of $\bK(s,t)$ over the lag variable $s$.
This integration does not commute with non-linear operators, such as the matrix logarithm, and the integral of products generally does not equal the product of integrals.
Consequently, the classical LIM does not analytically reconstruct the mean state.
Nevertheless, in practice, the estimated linear dynamics is sufficiently close to the mean state and reveals the dynamics-relevant information of the underlying system.
See also \cref{Table:Intercomparison}.


\section{Numerical Experiments} \label{Chap:NumExp}

In this section, we examine the performance of LIMs by applying these models to the dataset $\{x(t)\}$ sampled from periodic linear Markov system \cref{Eq:CS-LIM-Process}.
More preciously, with periodic linear dynamics $\bA(t)$ and diffusion $\bQ(t)$, a sample path is generated by the Euler approach with timestep $dt = 0.002$ from $T_0 = 0$ to $T_f = L$; then we make a sparse observation by taking $\Delta t = 5 \, dt = 0.01$ to obtain $\{x(t)\}$.
As the accuracy of models depends on the stochastic nature of the dataset, each experiment will be repeated $1024$ times and the consequent statistics (e.g. median) will be used to evaluate the models.
To quantify the performance, we use the relative error $e_\bX$ measured by Frobenius norm $\left\Vert \cdot \right\Vert_F$ for a matrix-valued quantity $\bX$ and the model output $X_\text{Model}$; that is,
\[ e_{\bX} = \frac{\left\Vert X_\text{Model} - \bX \right\Vert_F}{\left\Vert \bX \right\Vert_F}. \]
On the other hand, for a time sequence of matrices $\bX(t)$ and the model output $X_\text{Model}(t)$, the relative error $E_\bX$ is measured by the numerical integration of the following formula,
\begin{align} \label{Eq:Relative_Error_Formula}
    E_{\bX} \approx \frac{ \big( \int_0^1 \left\Vert X_\text{Model}(t) - \bX(t) \right\Vert_F^2 \, dt \big)^{\frac{1}{2}} }{ \big( \int_0^1 \left\Vert \bX(t) \right\Vert_F^2 \, dt \big)^{\frac{1}{2}} }.
\end{align}

\subsection{An ideal \texorpdfstring{$1$}{Lg}-dimensional case study} \label{Chap:1d-study}
We start with a $1$-dimensional case to demonstrate the lack of temporal structure of the classical LIM, the phase shift introduced in the original CS-LIM, and the noise issue of $l$-CS-LIM due to the noisy covariance function. 
The number of intervals and time-lag in $e$-CS-LIM are chosen to be $M = k = 10$ to distinguish the pointwise linear Markov approximation from the interval-wise version.
As an ideal case study, we suppose that the linear dynamics and diffusion fluctuate sinusoidally; that is, 
\begin{equation}\label{Eq:Periodic-Diffusion-Process}
    \frac{d}{dt} \bx = \big(1+ \ba \cdot \pi \sin(2 \pi t)\big) \overline{\bA}\bx + \sqrt{2 \cdot \big(1+\bb\cdot\pi\sin(2 \pi t)\big) \overline{\bQ}} \bmxi,
\end{equation}
where mean dynamics $\overline{\bA} = -1$, dynamical fluctuation intensity $\ba = 0.2$, mean diffusion $\overline{\bQ} = 1$, and diffusion fluctuation intensity $\bb = 0.3$.

\cref{Fig:Periodic-Diffusion} shows the numerical test for a sample path of \cref{Eq:Periodic-Diffusion-Process} with $T_f = 5000$.
It is apparent that the classical LIM does not reveal any temporal structure, but meanwhile, it successfully captures the mean dynamics and diffusion, showing the validity of linear Markov approximation. 
At the same time, even if the noisy nature of SDE, both the original CS-LIM and $e$-CS-LIM exhibit a clean temporal sinusoidal trend, since the information within an interval is collected, averaging out the noisy components, and a larger time step and time lag in the computation of $\partial_s C$ and $\Acslim$ lead to the robustness to the noise.
However, the original CS-LIM results are shifted to the left in both linear dynamics and diffusion while the $e$-CS-LIM results almost match the ground truth. 
On the other hand, we notice that the numerical covariance function $C(t) = K(0,t)$ used by $l$-CS-LIM exhibits a noisy nature, leading to an even more drastically oscillating first derivative if the finite difference method is applied directly \cite{Chapra2018,Chartrand2011,STICKEL2010}. 
If none of the filters is applied to the covariance function, the reconstructed profiles by $l$-CS-LIM, especially $\Aplim$, turn out to be noisy as well. 
Though seemingly ill-behaved, a simple application of convolution filtering reveals the correct temporal trend of dynamics; for instance, the solid blue line in \cref{Fig:Periodic-Diffusion} shows the result of the moving average filter (only time coordinate coincided with $\Tcslim$ are plotted) which exhibits the sinusoidal trend of $\bA(t)$.
Intriguingly, the reconstructed $\Qplim$ behaves more stable than $\Aplim$, probably due to the appropriate choice of finite difference method and the idealization of this study.

\begin{figure}[htbp]
\centering
    \centering
    \includegraphics[width=5in]{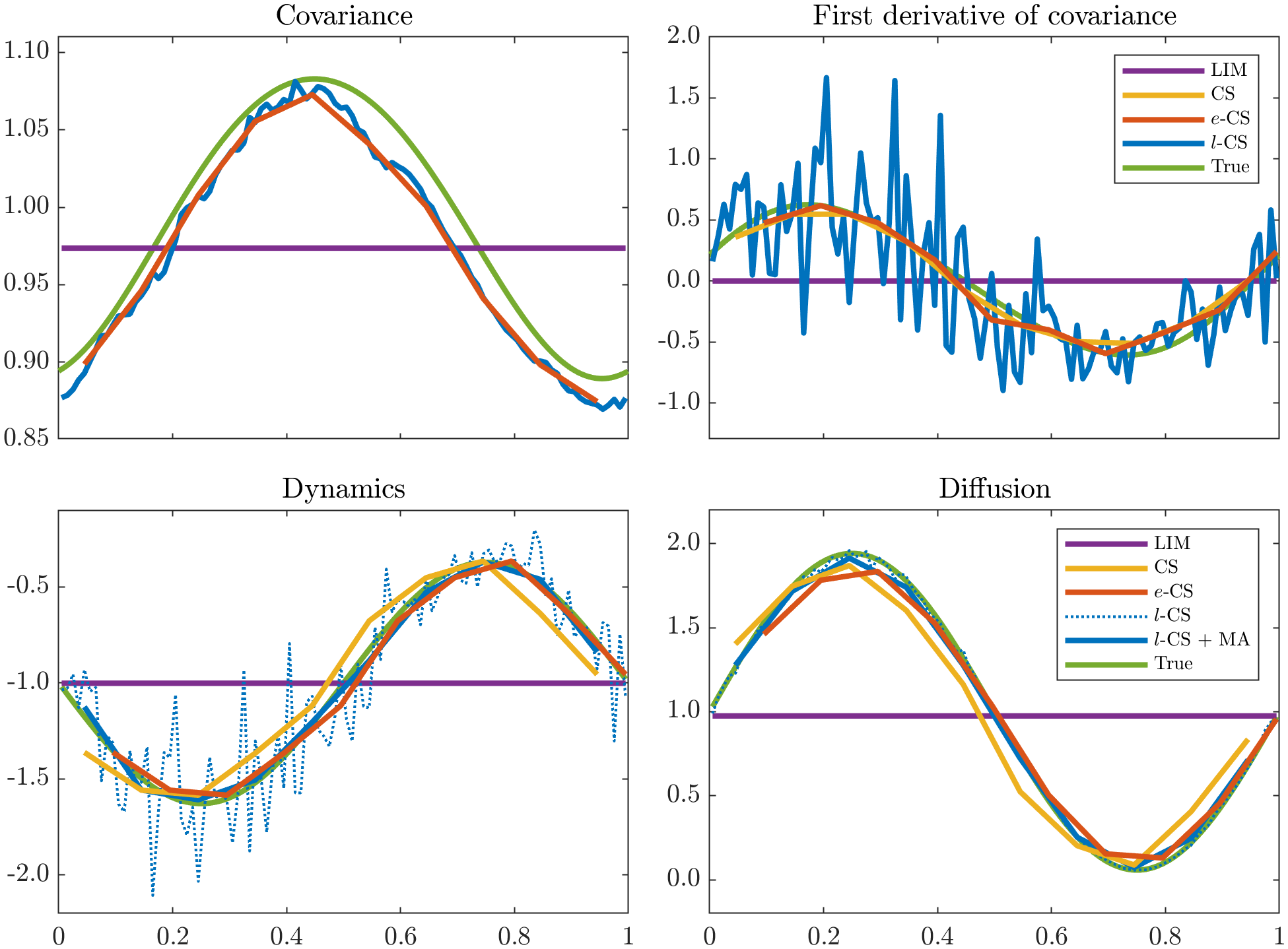}
    \caption{The ground truth and the LIM results for a sample path of \cref{Eq:Periodic-Diffusion-Process} with $T_f = 5000$. For the results of $l$-CS-LIM, the dashed lines are $\Aplim$ and $\Qplim$; the solid lines are $A_\text{MA}$ and $Q_\text{MA}$ sampled on $\Tcslim$.  }
    \label{Fig:Periodic-Diffusion}
\end{figure}


Now we quantitatively evaluate and compare the performance of each model.
As the use and choice of filters depend on the practical problem at hand, for a moment, no filter is applied.
To remove the influence of noise, especially in $l$-CS-LIM, we evaluate the performance of each model via the reconstructed mean dynamics $\overline{A_\text{Model}}$ and the fluctuation intensity $a_\text{Model}$ by a sine wave fitting that minimizes 
\begin{align*}
    \min_{\overline{A}, a, \phi} \left\Vert A_\text{Model}(t) - \overline{A} \cdot \big(1 + a \cdot \pi \sin(2 \pi (t+\phi) )\big) \right\Vert_F,
\end{align*}
and $\overline{Q_\text{Model}}$ and $b_\text{Model}$ by the same formula. 

\cref{Fig:Periodic-Diffusion_Accuracy} shows the distribution of the relative errors with different sampling sizes $T_f \in \{ 100, 1000, 5000 \}$. 
In general, regardless of $T_f$, the $l$-CS-LIM performs comparably to $e$-CS-LIM in terms of dynamics and achieves a lower relative error in diffusion.
We notice that the classical LIM exhibits fairly accurate mean values and the original CS-LIM already gives a decent result, indicating that both the linear Markov approximation and its interval-wise counterpart work well in our study. 
However, \cref{Fig:Periodic-Diffusion_Phase} demonstrates a consistent left phase-shifting ($\phi \approx 0.05 > 0$ for dynamics and $\approx 0.03 > 0$ for diffusion) in the original CS-LIM results, which is not obvious in the $e$-CS-LIM. Such difference in the solutions and their corresponding phase shifts can be explained by the different stencil configurations shown in \cref{Fig:Stencil}. In the original CS-LIM framework, the exponential fitting involves the dataset in the adjacent interval but the time coordinate is not adjusted accordingly, which creates a phase shift. Therefore, the proposed e-CS-LIM in this study, with refined stencil configuration, is more preferable for real-world applications to avoid the uncertainty introduced by phase shifts in solutions.

\begin{figure}[htbp]
\centering
    \centering
    \includegraphics[width=5in]{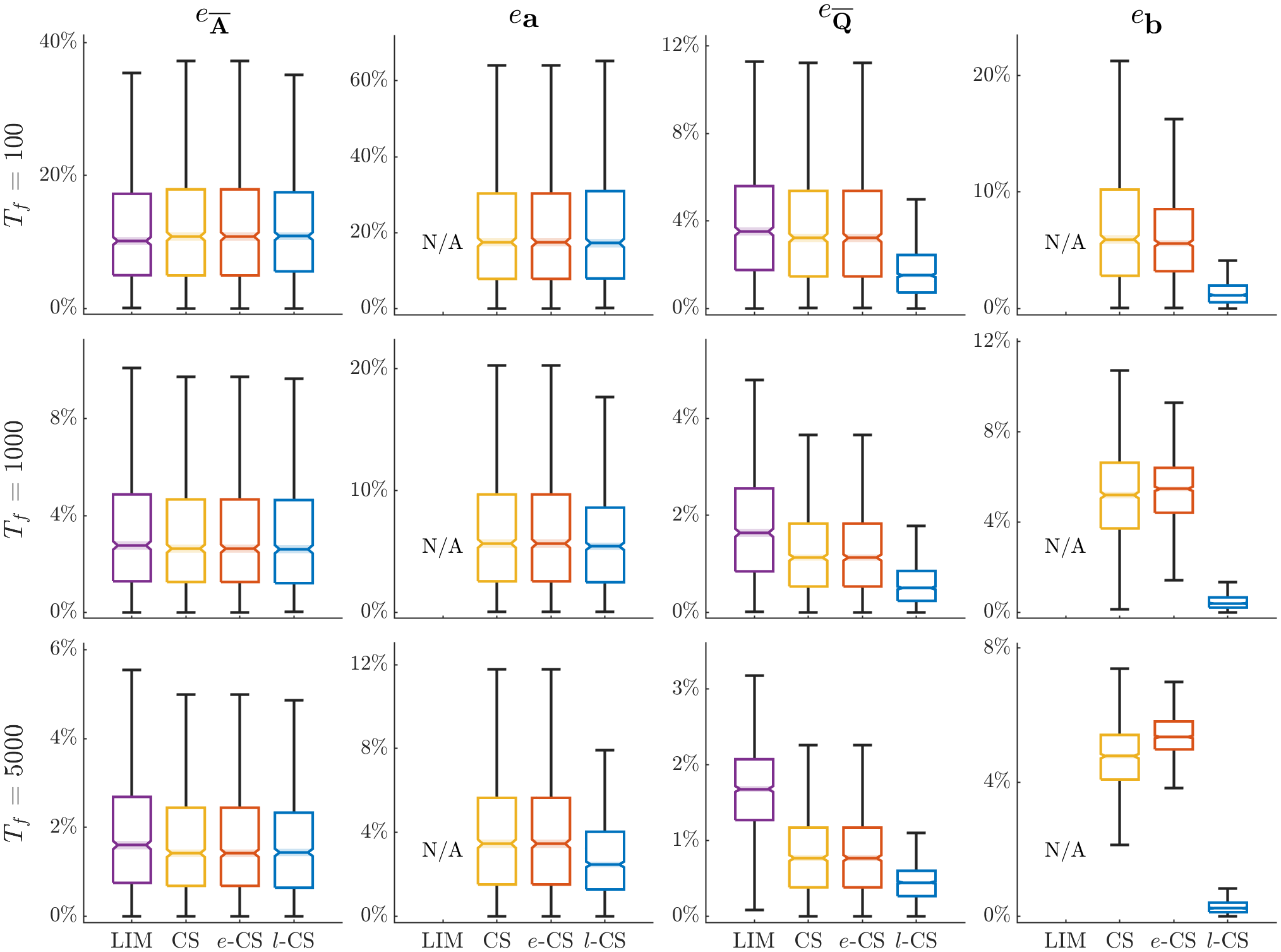}
    \caption{ The distribution of the relative errors of sine wave fitting for each model. }
    \label{Fig:Periodic-Diffusion_Accuracy}
\end{figure}

\begin{figure}[htbp]
\centering
    \centering
    \includegraphics[scale = 0.8]{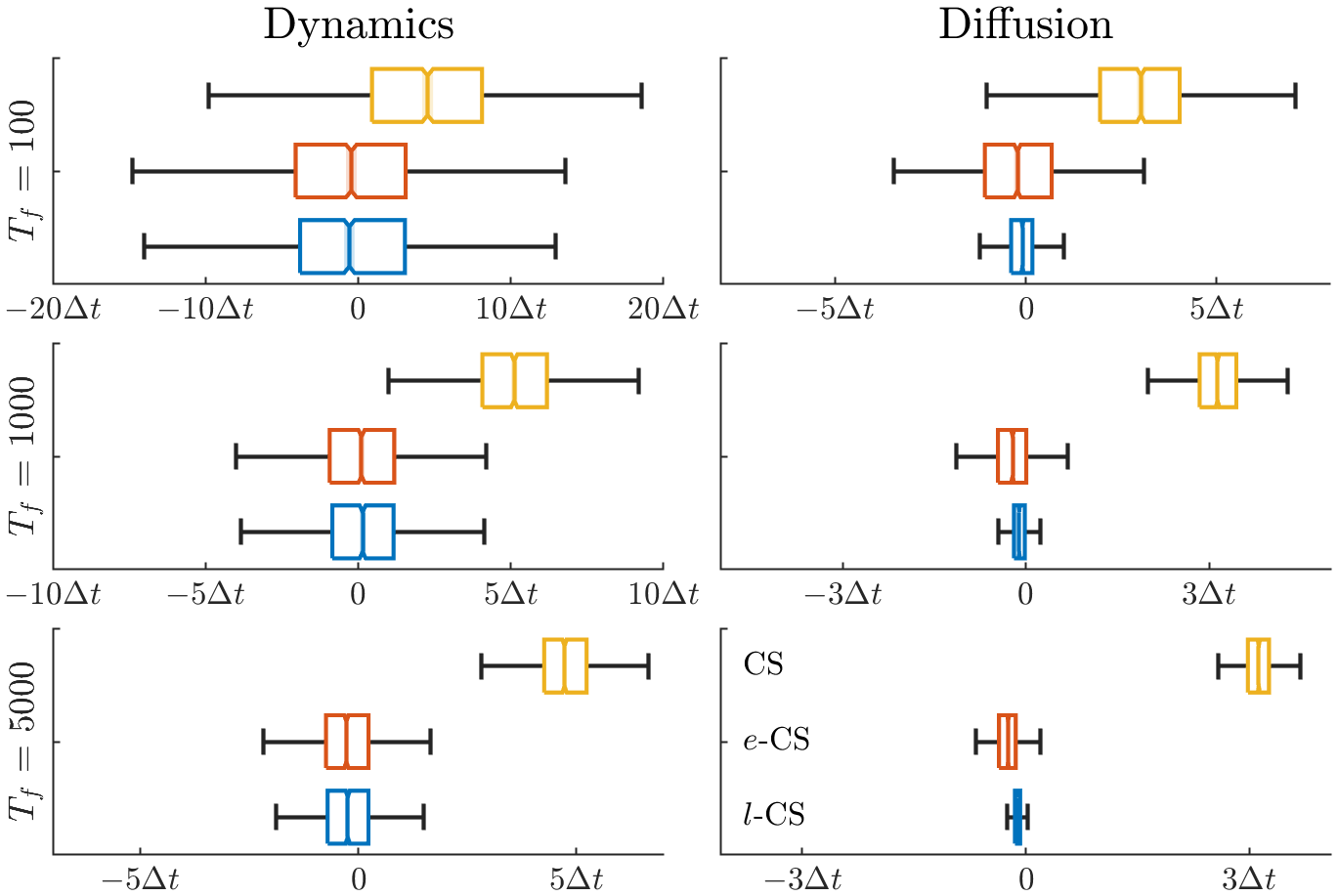}
    \caption{ The distribution of $\phi$ for each model. The CS-LIM, $e$-CS-LIM, and $l$-CS-LIM from the top to the bottom are represented by yellow, orange, and blue, respectively. }
    \label{Fig:Periodic-Diffusion_Phase}
\end{figure}

In practice, the temporal structures of $\bA(t)$ and $\bQ(t)$ are important, and on some occasions, we have no prior knowledge of the shape of fluctuation, meaning that noise reduction should be applied on $l$-CS-LIM results. 
\cref{Fig:Periodic-Diffusion-Filter} demonstrates the application of a variety of filters (MA: moving average, LP: low-pass, and GW: Gaussian weight) on the reconstruction $\Aplim$, and each resulting curve $A_\text{Filter}$ reveals the underlying sinusoidal behavior of linear dynamics and diffusion. 
In addition, we note that the filtered dynamics $A_\text{Filter}(t)$ agree in a neighborhood of $t = 0.75$ where the random forcing is limited.
We evaluate the goodness of the filters by the relative errors $E_{l}$ (no filter), $E_{\text{MA}}$, $E_{\text{LP}}$, and $E_{\text{GW}}$, and summarize the results in \cref{Fig:Periodic-Diffusion-Filter}. 
The numerical test for this ideal study implies that applying filters can effectively reduce the relative error by more than half.
Indeed, the moving average filter already yields a decent reconstruction.

\begin{figure}[htbp]
\centering
    \centering
    \includegraphics[width=5in]{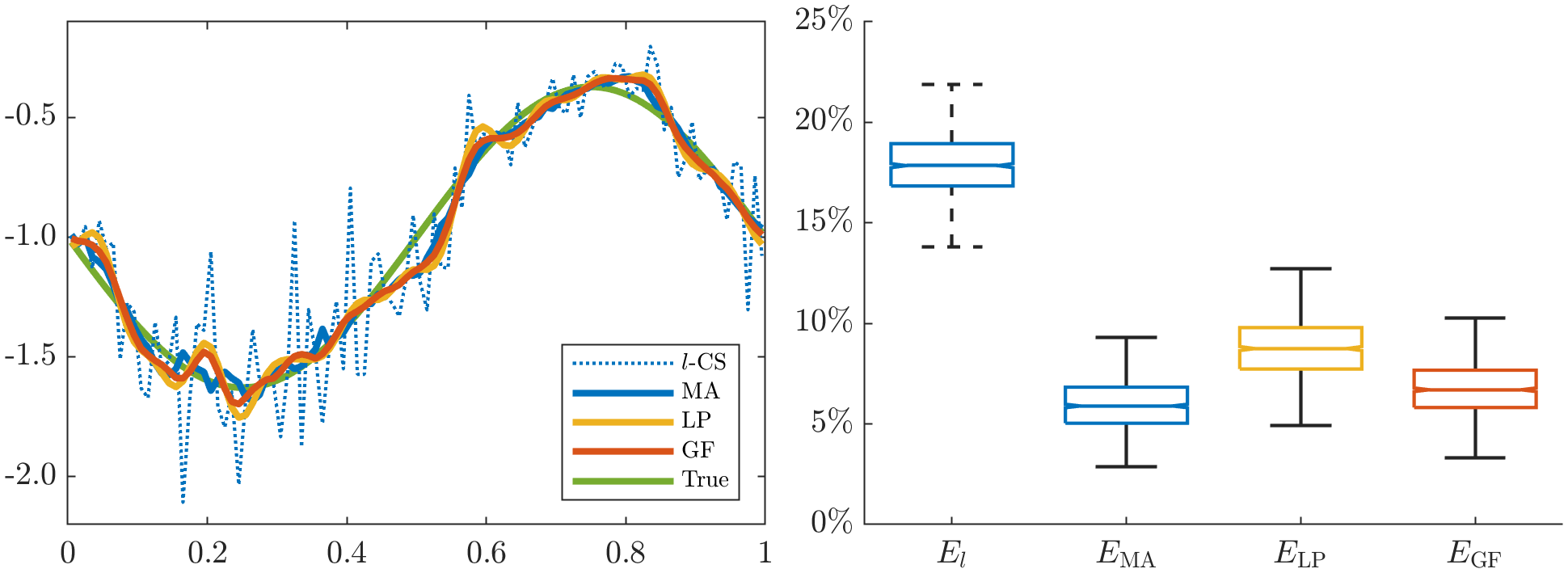}
    \caption{The reconstruction by $l$-CS-LIM and the application of filters.}
    \label{Fig:Periodic-Diffusion-Filter}
\end{figure}


\subsection{A higher-dimensional case study} \label{Chap:Higher-Dimension}

With the same setup as in the previous section, we evaluate the performance of LIMs in higher-dimensional cases without assuming the preknowledge of the sinusoidal fluctuation, since in real-world applications, it may be difficult to know the shape of fluctuation in advance.
The underlying dynamics and diffusion are set to fluctuate sinusoidally with the fluctuation intensities $\ba = 0.2$ and $\bb = 0.3$ as in \cref{Chap:1d-study}, the mean dynamics $\overline{\bA}$ randomly generated such that the magnitude of each entry is smaller than $5$ and each eigenvalue of $\bA$ has a negative real part, and the mean diffusion $\overline{\bQ}$ randomly generated such that it is positively definite.

To compare $l$-CS-LIM with the original CS-LIM and $e$-CS-LIM on the same basis, we take the average of $\Aplim$ and $\Qplim$ on each interval applied in the original CS-LIM and $e$-CS-LIM, and use the relative errors of the resulting curves $\widetilde{\Aplim}$ and $\widetilde{\Qplim}$ as the performance of $l$-CS-LIM.


\cref{Fig:Higher_Dimensional} shows the relative $L^2$-error $E_\bA$ and $E_\bQ$ for each model.
For small sampling size $T_f = 100$, the classical LIM may outperform the other models, but this is merely due to the fact that it adopts a safe strategy by neglecting the temporal structure.
As $T_f$ grows, its relative $L^2$-errors are stuck at a certain level while the relative errors of the CS-LIMs decrease.
As in the $1$-dimensional case, the classical LIM already yields fairly low relative errors $e_{\overline{\bA}}$ and $e_{\overline{\bQ}}$ in mean values, but CS-LIMs may indeed perform better, as shown in \cref{Fig:Higher_Dimensional_mean} since the linear Markov approximation may be an oversimplification of a periodic dynamical system.
On the other hand, the original CS-LIM makes a great improvement compared to the classical LIM, especially when the sampling is sufficient, by taking the temporal structure into consideration. 
However, due to the phase shift, its $L^2$-error $E_\bA$ and $E_\bQ$ eventually reach limitation.
For $T_f = 5000$, the original CS-LIM does not perform as superior as the other CS-LIMs.

In general, the $e$-CS-LIM performs better among all CS-LIMs in linear dynamics, as shown in \cref{Fig:Higher_Dimensional}, since it is designed to be more robust to noisy components. 
This also demonstrates the validity of interval-wise linear Markov approximation in higher dimensional cases.
On the other hand, the $l$-CS-LIM is subject to noise effect which is a classical shortcoming of using the finite difference method with a small timestep $\Delta t$, leading to a slightly larger $E_\bA$ even after taking the average over each interval. 

Overall, $l$-CS-LIM is better than $e$-CS-LIM for diffusion, which can be attributed to the order of operations: the $l$-CS-LIM first computes the time-dependent profiles $\Aplim$ and $\Qplim$, then averaging out the oscillation over each interval, while the $e$-CS-LIM averages the data first, then performs the computation. 
Though subtle, $e$-CS-LIM implicitly replaces the integral of products with the product of integrals, which is equivalent to use the following formula
\begin{align} \label{Eq:Order}
    \int_{I_j} \frac{d}{dt} \bC(t) dt - \big(  \int_{I_j} \bA(t) \, dt \, \int_{I_j} \bC(t) \, dt + \int_{I_j} \bC \, dt \, \int_{I_j} \bA^T \, dt + 2 \int_{I_j} \bQ \, dt \big)  \approx 0,
\end{align} 
where $I_j$ is the $j$-th interval, to compute $\Qcslim(j) \approx \int_{I_j} \bQ \, dt$, causing a larger $E_\bQ$ compared to the $l$-CS-LIM results.
The relationship and intercomparison among LIMs are summerized in \cref{Table:Intercomparison}.

Finally, \cref{Fig:More_Oscillation} shows the convergence of $e$-CS-LIM to $l$-CS-LIM. For a fixed $\Delta t$ and $T_f$, the relative difference\footnote{ The relative difference of two time-series of matrices is defined by a similar formula as \cref{Eq:Relative_Error_Formula}. } between the outputs of both time-dependent dynamics and diffusion decreases as the number of subdivision $M = 1/\rho$ increases. 
The difference further decreases as time-span $T_f$ increases from $100$ to $5000$. 
However, we emphasize that the convergence does not mean that the relative error of $e$-CS-LIM monotonically decreases since $l$-CS-LIM does not always achieve a superior performance due to the noise effect.

Before moving to the application, we note that the above discussion holds provided the sampling points within a period are sufficient (i.e., $\Delta t \ll 1$).
However, if the sampling points are not enough (or equivalently, the sampling interval is large), $l$-CS-LIM no longer outperforms $e$-CS-LIM in either dynamics or diffusion.
This may be attributed to the finite-difference-based nature of $l$-CS-LIM, and once again implying the effectiveness of interval-wise linear Markov approximation used by $e$-CS-LIM.

\begin{sidewaysfigure}[ht]
    \includegraphics[width=\textwidth]{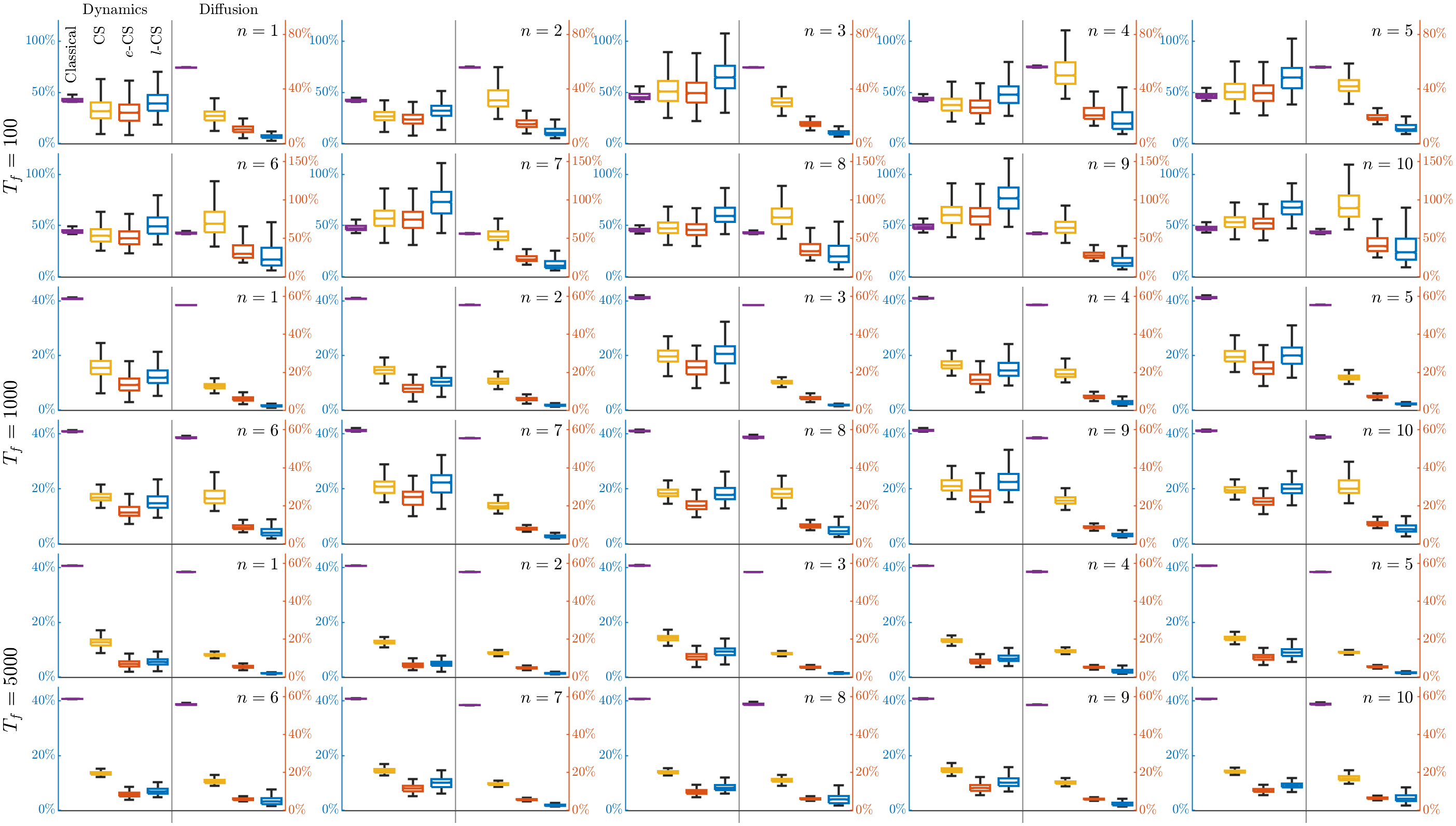}
    \caption{The distribution of the relative $L^2$-errors $E_\bA$ (left-panel) and $E_\bQ$ (right-panel) for each model across dimension $n$ and sampling size $T_f$ over $1024$ trials. In each panel, the classical LIM, original CS-LIM, $e$-CS-LIM, and $l$-CS-LIM from the left to the right are represented by purple, yellow, orange, and blue, respectively.}
    \label{Fig:Higher_Dimensional}
\end{sidewaysfigure}

\begin{sidewaysfigure}[ht]
    \includegraphics[width=\textwidth]{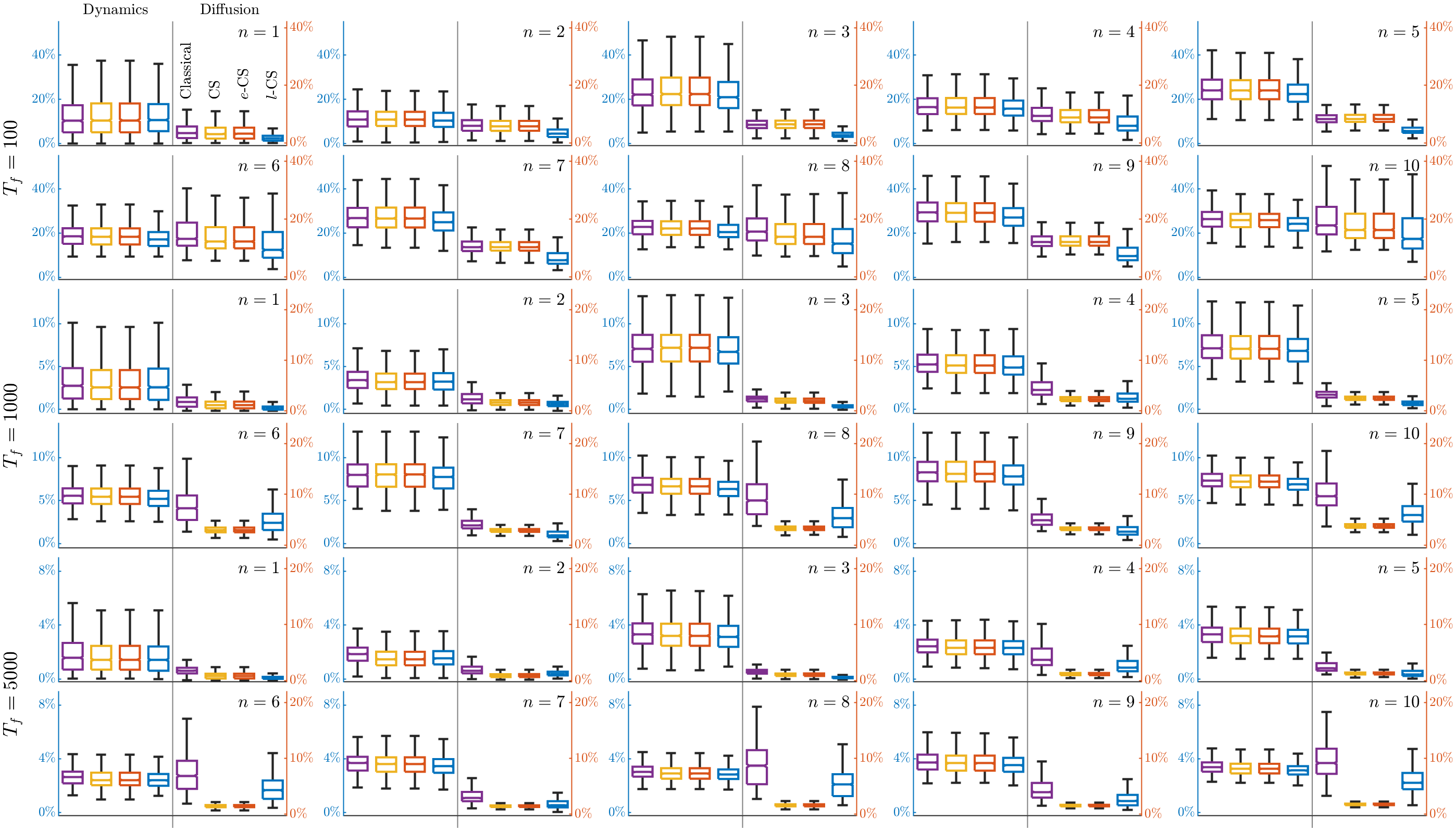}
    \caption{ The distribution of the relative errors $e_{\overline{\bA}}$ (left-panel) and $e_{\overline{\bQ}}$ (right-panel) for each model across dimension $n$ and sampling size $T_f$ over $1024$ trials. In each panel, the classical LIM, original CS-LIM, $e$-CS-LIM, and $l$-CS-LIM from the left to the right are represented by purple, yellow, orange, and blue, respectively.}
    \label{Fig:Higher_Dimensional_mean}
\end{sidewaysfigure}

\begin{sidewaysfigure}[ht]
    \includegraphics[width=\textwidth]{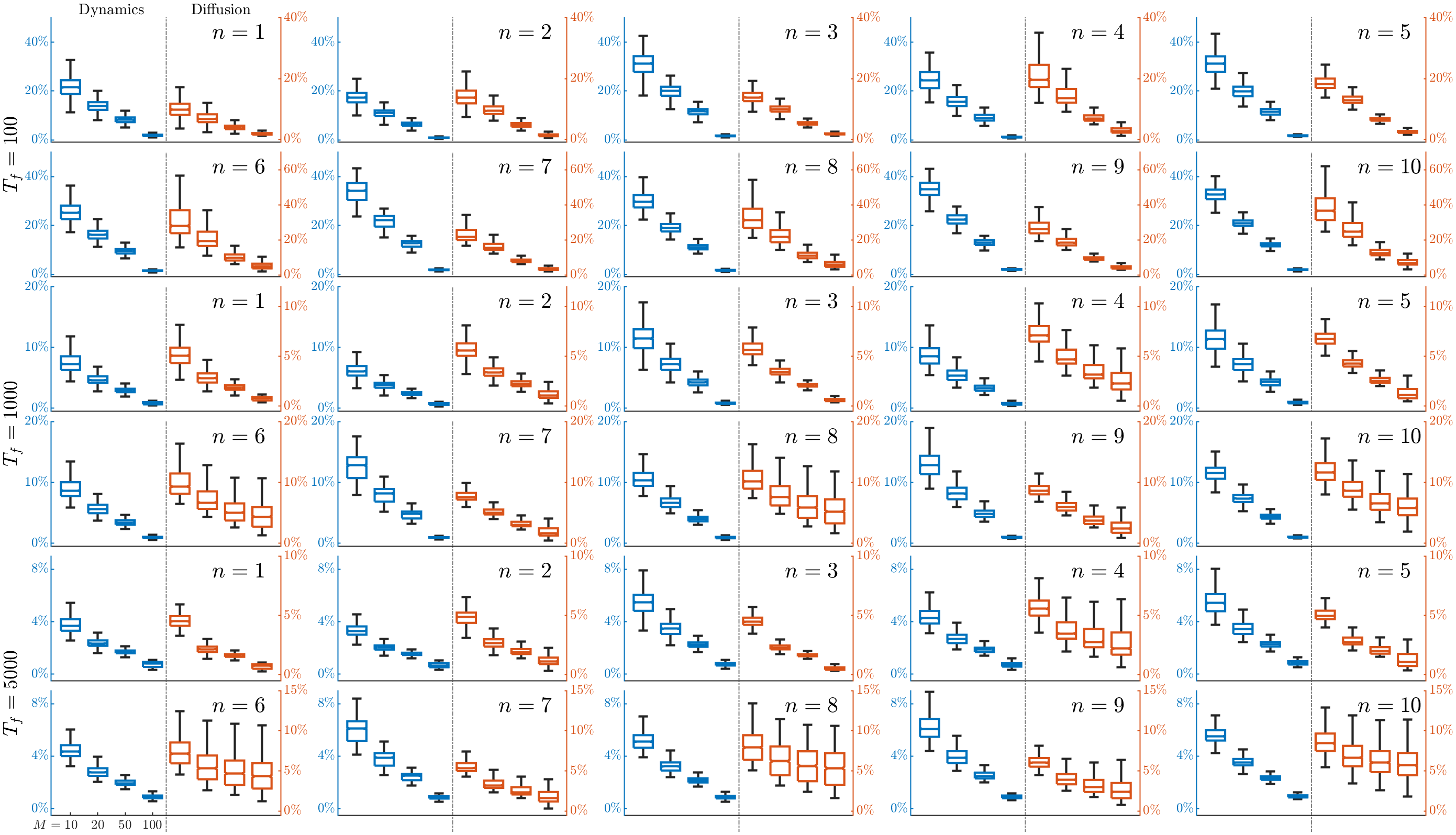}
    \caption{ The distribution of the relative difference of the $e$-CS-LIM and $l$-CS-LIM's output across dimension $n$ and sampling size $T_f$ over $1024$ trials (dynamics: left-panel and blue; diffusion: right-panel and orange). The time lag $\rho$ is chosen to be the reciprocal of the number of subdivision $M$. }
    \label{Fig:More_Oscillation}
\end{sidewaysfigure}

\begin{sidewaystable}
\footnotesize
\centering
\caption{Intercomparison of linear inverse models}
\footnotesize
\centering
\begin{threeparttable}
\begin{tabular}{ c c c c c }
    \toprule
    {} & {Classical-LIM} & {Original CS-LIM} & {$e$-CS-LIM} & {$l$-CS-LIM}  \\
    \midrule
    \makecell{Mathematical \\ Background} & \makecell{Analytic correlation \\ function \& FDR} & \multicolumn{2}{c}{\makecell{Application of classical LIM \\ over each interval \& periodic FDR}} & \makecell{The right derivative of \\ correlation function \& periodic FDR} \\
    \midrule
    \makecell{Relationship with \\ linear Markov system} & {Linear Markov approximation} & \multicolumn{2}{c}{Interval-wise linear Markov approximation} & {Pointwise linear Markov approximation} \\
    \midrule
    \makecell{Estimating the \\ dynamics by} & {Exponential fitting} & \multicolumn{2}{c}{Exponential fitting} & {Linear fitting} \\
    \midrule
    \makecell{Performance on the \\ periodic LMS\tnote{\romannumeral1}} & \makecell{Only providing mean \\ state information} & \makecell{Better than classical LIM \\ but not optimal} & {Best in dynamics} & {Best in diffusion} \\
    \midrule
    Pros. & {Sufficient sampling data} & {--} & {Robust to noise} & {An analytic inverse model} \\
    \midrule
    Cons. & {No temporal structure being identified} & {A phase shift} & {Relying on approximation} & {Subject to noise} \\
    \bottomrule
\end{tabular} \label{Table:Intercomparison}
\begin{tablenotes}
    \item Under the condition of sufficient sampling data.
    \item[\romannumeral1] Periodic LMS = Periodic linear Markov system \cref{Eq:CS-LIM-Process}. 
\end{tablenotes}
\end{threeparttable}
\end{sidewaystable}

\subsection{A real-world example: Ni\~{n}o 3.4 SST Index}

The Ni\~{n}o 3.4 SST index, defined as the temperature average over 5S-5N and 170-120W, is a real-world cyclostationary monthly time-series data whose anomaly is widely used to monitor and predict El Ni\~{n}o and La Ni\~{n}a events \cite{Capotondi2015,Okumura2019}.
Simply speaking, the SST variability can be considered as the superposition of atmospheric variability (noise perturbation) and subsurface processes (background state) \cite{Deser2010,Lou2020}, and hence can be viewed as a periodically driven stochastic process. 
Previous studies have emphasized the seasonal variations of ENSO \cite{Cane1986,Chen2021,Chen2022,Levine2015,Shin2021,Webster1992}. 
Still, only a few of them quantified the relative roles of the seasonal variations of the predictable SST dynamics and unpredictable stochastic forcing \cite{Shin2021}.
Therefore, we apply the $e$-CS-LIM and $l$-CS-LIM to the anomaly of the Ni\~{n}o 3.4 SST index from $1884$ to $2020$ to quantify the linear time-dependent dynamics and random forcing. 
However, since the ENSO criteria are complicated and are beyond the scope of this article, we focus on the extreme ENSO peaks. 
These peaks are defined by the time at which the anomaly is at its maximum over the preceding and following $6$ months with $\left\vert \text{anomaly} \right\vert \ge 2$ $^\circ$C.
We denote the number of extreme ENSO peaks as \#EEP.
\Cref{Fig:Nino_Data} shows the observed SST anomaly and the occurrence of extreme ENSO peaks.

\cref{Fig:Nino_P_LIM} reveals the seasonal cycles of dynamics and diffusion: a sinusoidal shape for linear dynamics while a bimodal profile centered in spring and mid-autumn for diffusion.
The reconstructed profiles indicate that during the late spring and summer, though the background oceanic system pulls the index back to the equilibrium, the atmospheric system intends to determine the phase (either positive, negative, or neutral\footnote{close to zero}) of the upcoming ENSO peak though random forcing.
Once the phase is determined, the oceanic system then acts as a driving force for the development of ENSO peaks in autumn, and through atmospheric forcing, there is a chance for the anomaly to surpass the threshold of $\pm 2$ $^\circ$C, leading to potential extreme ENSO peaks in winter and early spring.
During March and April, the strong dynamical restoring forcing pulls the system into its equilibrium state, while the random forcing also strongly drives the system apart from its mean state, making it difficult to predict whether the index will go up or down in the upcoming months, which may be related to the so-called ENSO prediction barrier \cite{Cane1986,Levine2015,Webster1992}. 

For each model, we generate a $1024$-member ensemble of $137$ years for SST anomaly by numerically integrating $A_\text{Model}$ and $Q_\text{Model}$ with a timestep $dt = 0.001$ (year) and taking the mean value over a month as the representative, and summarize the results in \cref{Fig:Nino_Reconstruct}.
Both $e$-CS-LIM and $l$-CS-LIM consistently exhibit \#EEPs around $20$ times in the median over a $137$-year period, and such extreme peaks are more likely to appear in winter than summer, as the observed ENSO phase locking \cite{Chen2021,Chen2022}. 
Though the anomaly of Ni\~{n}o 3.4 SST index is merely an indicator for ENSO, we emphasize that CS-LIMs can be applied to higher-dimensional climate variables, providing spatialtemporal coherent information of a complex system.

\begin{figure}[htbp]
\centering
    \begin{subfigure}[b]{1\textwidth}
    \centering
    \includegraphics[width=3.0in]{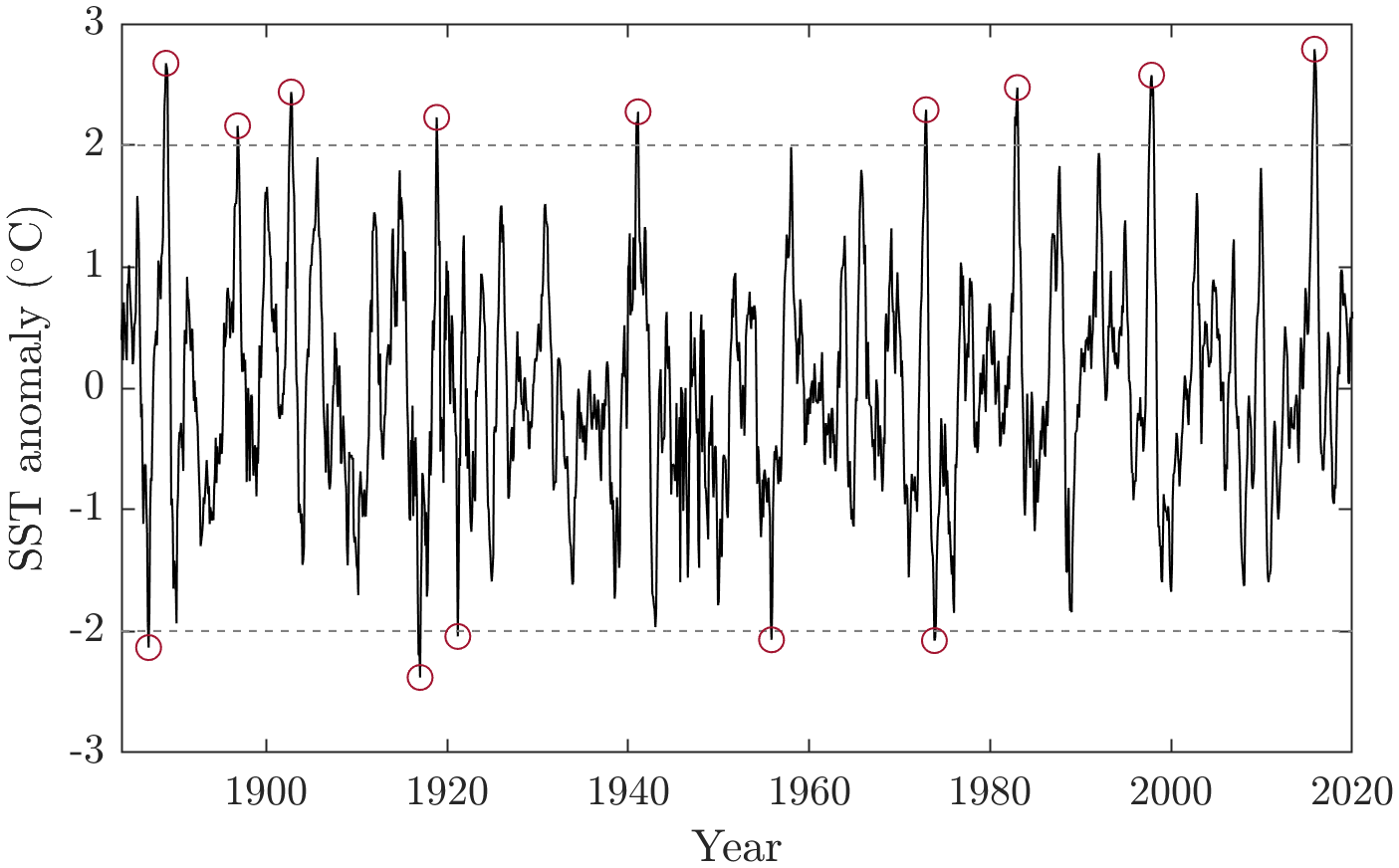}
    \caption{The anomaly of the Ni\~{n}o 3.4 SST index from 1884 to 2020 (with the seasonal cycle and global warming trend removed). The circle indicates the occurrence of an extreme ENSO peak and the dashed lines specify the $\pm 2^\circ$C threshold. } \label{Fig:Nino_Data}
    \end{subfigure}
    \begin{subfigure}[b]{1\textwidth}
    \centering
    \includegraphics[width=4.0in]{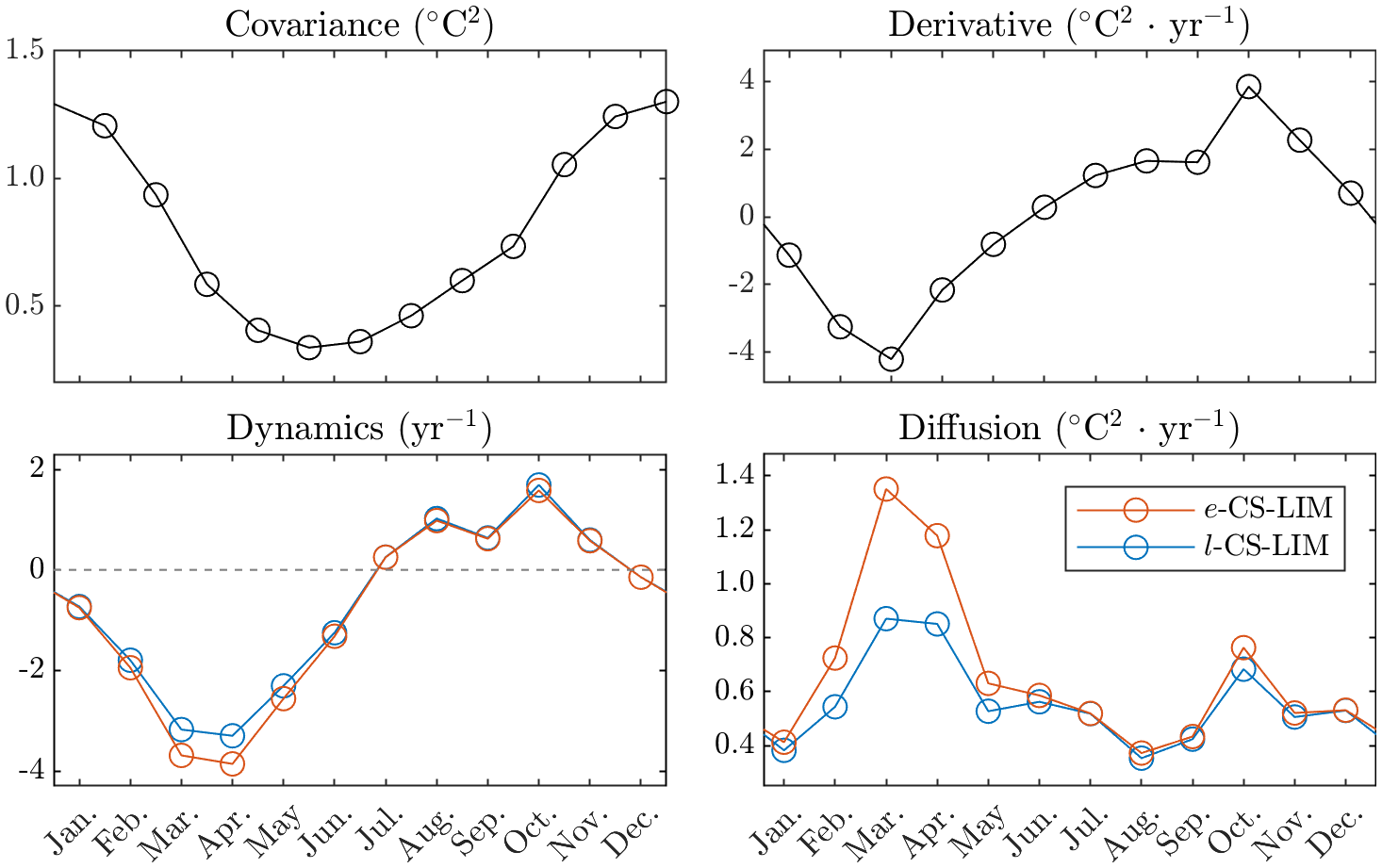}
    \caption{ The upper panels show the covariance function and its first derivative for the SST anomaly. The lower panels demonstrate the model results. The $x$-label represents the first day of each calendar month. Notice that the covariance and the $x$-label are alternating. No filter is applied.} \label{Fig:Nino_P_LIM}
    \end{subfigure}
    \begin{subfigure}[b]{1\textwidth}
    \centering
    \includegraphics[width=3.2in]{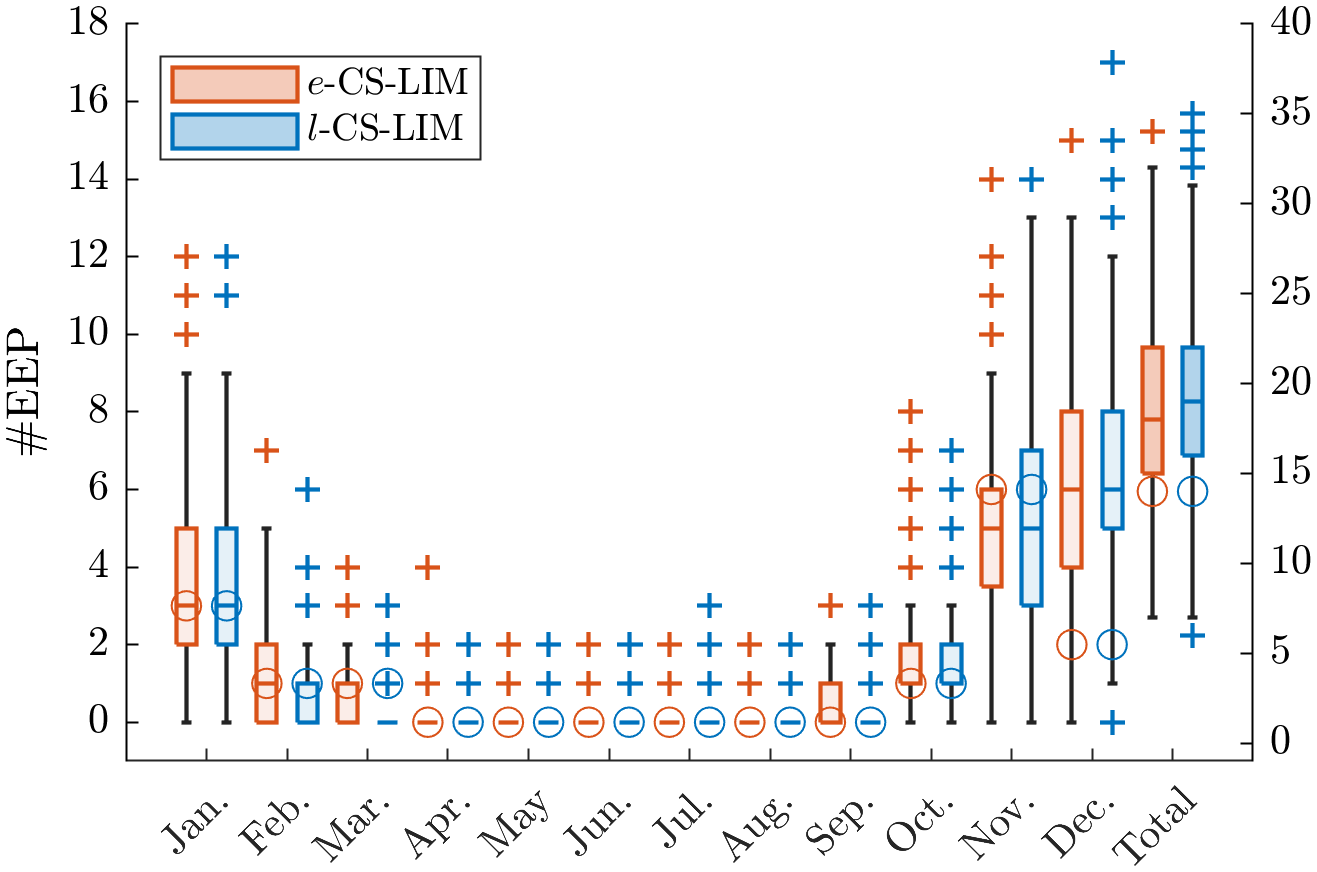}
    \caption{ The statistics of extreme ENSO peaks. 
    The circle indicates the observed \#EEP in a given month (left axis) over $1884$ to $2020$, and the total \#EEP over the whole $137$ years (right axis). The box plots show the \#EEP statistics for the ensembles. 
    The marker $+$ is the outlier.
    } \label{Fig:Nino_Reconstruct}
    \end{subfigure}
    \caption{The numerical study of the Ni\~{n}o 3.4 SST index anomaly.} \label{Fig:Nino}
\end{figure}


\section{Conclusion}

In this article, we study the family of CS-LIMs, a class of linear inverse models that estimates linear dynamics and random forcing of cyclostationary time-series data by approximating the underlying complex system with a periodic linear Markov system.
In particular, the original CS-LIM is optimized to the $e$-CS-LIM, and a novel linear inverse model called $l$-CS-LIM is proposed.
We have discussed the mathematical background of $e$-CS-LIM and showed that for a periodic linear Markov system, the time-dependent dynamical matrix at each instant is characterized by the right derivatives of the correlation function in the lag variable, which serves as the ground of $l$-CS-LIM. 
Moreover, the $e$-CS-LIM is built upon the interval-wise linear Markov approximation while in the $l$-CS-LIM case, it amounts to a pointwise linear Markov approximation.
We have also shown that $e$-CS-LIM and $l$-CS-LIM are closely related in the sense that the latter can be viewed as an instantaneous version of the former. 

The numerical experiments have shown that the classical LIM fails to capture the temporal trend of the system; the original CS-LIM, though greatly improves the classical LIM results, introduces an artificial phase shift.
Meanwhile, the $e$-CS-LIM effectively picks up the correct phase and reaches a better performance in linear dynamics among CS-LIMs in terms of relative $L^2$-error, verifying the validity of interval-wise linear Markov approximation in practice. 
In addition, the $e$-CS-LIM is designed to be more robust to noise. 
On the other hand, though based on analytic formulas, the $l$-CS-LIM is subject to noise, leading to an oscillating output, especially in dynamics. 
Nevertheless, an application of the moving average filter effectively removes the noise effect, revealing the temporal structure of the underlying system.
At the same time, as not excessively relying on approximation, the $l$-CS-LIM shows slightly worse but comparable results in dynamics and exhibits a significantly lower relative $L^2$-error in diffusion.

In principle, the CS-LIMs can be applied to a wide range of cyclostationary time series. 
In this study, we have applied the CS-LIMs to the real-world Ni\~{n}o SST 3.4 index, data for investigating El Ni\~{n}o and La Ni\~{n}a phenomena in climate sciences.
The CS-LIM results indicate a strong seasonal dependency, especially in dynamics, and can well explain the development of extreme ENSO peaks, consistent with our current understanding.
Furthermore, we have re-integrated the CS-LIMs outputs, and seen that the reproduced Ni\~{n}o index captures the occurrence of extreme ENSO peaks, in agreement with observation.
Though this ENSO study is a simplification of the complex earth system and merely serves as a demonstration of the potential of CS-LIMs, we believe that both $e$-CS-LIM and $l$-CS-LIM will lead to different insights into our understanding of the complex climate system and other fields of study. 







\bibliographystyle{abbrv}
\bibliography{Reference}
\end{document}